\documentclass[12pt,a4paper]{amsart}
\usepackage{amssymb,amsmath}
\usepackage{cancel}
\usepackage{palatino}
\usepackage{longtable}
\usepackage[hmargin=3cm,vmargin=3cm]{geometry}
\usepackage{color}
\numberwithin{equation}{section}
\newtheorem{theorem}{Theorem}[section]     

\newtheorem{defi}[theorem]{Definition}
\newtheorem{thm}[theorem]{Theorem}
\newtheorem{prop}[theorem]{Proposition}

\newtheorem{rmk}[theorem]{Remark}

\newtheorem{es}[theorem]{Example}

\def\d{\partial}

\def\f{\frac}

\newcommand{\eqa}{\begin{eqnarray}}
\newcommand{\eeqa}{\end{eqnarray}}
\newcommand{\beq}{\begin{equation}}
\newcommand{\eeq}{\end{equation}}

\newcommand{\F}{{\cal F}}

\numberwithin{equation}{section}

\begin{document}
\title{Regular non-semisimple Dubrovin-Frobenius manifolds}

\author{Paolo Lorenzoni}
\address{P.~Lorenzoni:\newline Dipartimento di Matematica e Applicazioni, Universit\`a di Milano-Bicocca, \newline
Via Roberto Cozzi 53, I-20125 Milano, Italy and INFN sezione di Milano-Bicocca}
\email{paolo.lorenzoni@unimib.it}

\author{Sara Perletti}
\address{S.~Perletti:\newline Dipartimento di Matematica e Applicazioni, Universit\`a di Milano-Bicocca, \newline
Via Roberto Cozzi 53, I-20125 Milano, Italy and INFN sezione di Milano-Bicocca}
\email{s.perletti1@campus.unimib.it}

\maketitle

\begin{abstract}
We study regular non semisimple Dubrovin-Frobenius manifolds in dimension $2,3,4$. We focus on the case where the Jordan canonical
 form of the operator of multiplication by the Euler vector field has a single Jordan block. Our results rely on the existence of special local coordinates
 introduced in \cite{DH} for regular flat F-manifolds with Euler vector field. In such coordinates the invariant metric of the Dubrovin-Frobenius manifold 
 takes a special form which is the starting point of our construction.
\end{abstract}

\tableofcontents

\section*{Introduction}
\begin{sloppypar}
 Dubrovin-Frobenius manifolds has been introduced by Boris Dubrovin as a coordinate-free reformulation of the so called Witten-Dijkgraaf-Verlinde-Verlinde (WDVV)-equations of two-dimensional topological field theories (see \cite{D96}) and play an important role  in many areas of mathematics (quantum cohomology, Gromov-Witten theory, singularity theory, integrable PDEs etc). Some constructions in the theory of Dubrovin-Frobenius manifolds rely on an additional assumption: the existence of a holonomic frame of idempotents. Dubrovin-Frobenius manifolds having this property are called semisimple or massive since in a physical context they correspond to massive perturbations of two-dimensional topological field theories. Semisimple Dubrovin-Frobenius manifolds are characterized by the existence of a special set of local coordinates, called Dubrovin canonical coordinates or simply canonical coordinates, reducing the structure constants of the product to a constant canonical form. A generalization of canonical coordinates  in the non semisimple regular case was found  by David and Hertling in \cite{DH}. David-Hertling canonical coordinates depend on the Jordan normal form of the operator of multiplication by the Euler vector field. In this paper using these coodinates we construct explicit examples of non semisimple regular Dubrovin-Frobenius manifolds in the case of a single Jordan block. 
\end{sloppypar}

The paper is organized as follows: In section $1$ we recall the definition of Dubrovin-Frobenius manifold and some known results in the semisimple case.
 In Section 2 we introduce David-Hertling canonical coordinates for regular non semisimple Dubrovin-Frobenius manifolds and some general properties
 of the invartiant metric in such coordinates. In Section 3 we focus on the case of a single Jordan block in dimension $2,3$ and $4$.

\section{Dubrovin-Frobenius manifolds: the semisimple case}  
Following Dubrovin \cite{D96} we introduce the notion of Dubrovin-Frobenius manifold.
\begin{defi}
A Dubrovin-Frobenius manifold $M$ is a manifold equipped with a metric $\eta$, a commutative associative product $\circ$ on the tangent space with unit $e$ and a second distinguished vector field $E$ called the Euler vector field satisfying the following conditions
\begin{itemize}
\item Invariance of the metric:
\begin{equation}\label{EqI}
\eta_{il}\,c^l_{jk}=\eta_{jl}\,c^l_{ik}
\end{equation}
\item Flatness of the metric:
\begin{equation}\label{EqXI}
R^m_{ijk}=\partial_j\Gamma^m_{ik}-\partial_i\Gamma^m_{jk}+\Gamma^s_{ik}\Gamma^m_{sj}-\Gamma^s_{jk}\Gamma^m_{is}=0
\end{equation}
\item Symmetry of $\nabla c$:
\begin{equation}\label{EqII}
\nabla_ic^l_{jk}=\nabla_jc^l_{ik}
\end{equation}
\item Constancy of $e$:
\begin{equation}\label{EqIII}
\nabla_ie^k=0
\end{equation}
\item Homogeneity conditions:
\begin{equation}\label{homog}
\mathcal{L}_E c^i_{jk}=c^i_{jk},\qquad \mathcal{L}_E e^i=-e^i,\qquad \mathcal{L}_E\eta_{ij}=(2-d)\eta_{ij}
\end{equation}
\end{itemize}
for some constant $d$. Here $\nabla$ denotes the Levi-Civita connection associated to $\eta$ and $\mathcal{L}_Z$ denotes the Lie derivative along a vector field $Z$. 
\end{defi}

From the axioms above it follows that in flat coordinates for the metric the structure constants of the product can be written in terms of the third order partial derivatives of a function $F$ called the prepotential of the Dubrovin-Frobenius manifold:
\[c^i_{jk}=\eta^{il}\d_l\d_j\d_k F.\]
By construction the function $F$ is a solution of Witten-Dijkgraaf-Verlinde-Verlinde (WDVV) equations \cite{W,DVV}.

\begin{rmk}
The manifold $M$ in the above definition  is  a real or complex $n$-dimensional manifold. In the first case all the geometric data are supposed to be smooth.
 In the latter case $TM$ is intended as the holomorphic tangent bundle and all the geometric data are supposed to be holomorphic. 
\end{rmk}

\begin{rmk}
Since the components of the metric and of the unit vector field are constant in flat coordinates we clearly have
\begin{equation}\label{EqVIII}
\mathcal{L}_e\eta_{ij}=0.
\end{equation}
\end{rmk}
A point $p\in M$ of an $n$-dimensional Dubrovin-Frobenius manifold is called \textit{semisimple} if $T_pM$ has a basis of idempotents $\pi_1,\dots,\pi_n$ satisfying $\pi_k \circ \pi_l = \delta_{k,l} \pi_k$. Semisimplicity at a point is an open property on $M$: locally around a semisimple point one can choose coordinates $u^i$ such that $\frac{\d}{\d u^k}\circ\frac{\d}{\d u^l}=\delta_{k,l}\frac{\d}{\d u^k}$. These coordinates are called {\it canonical coordinates}.

Due to \eqref{EqI}, in canonical coordinates the metric $\eta$ becomes diagonal: $\eta_{ij}=H_i^2\delta_{ij}$. Let us introduce the \emph{Ricci rotation coefficients} $\beta_{ij}:=\f{\d_j H_i}{H_j}$, $i\ne j$. In the case of Dubrovin-Frobenius manifolds the rotation coefficients are symmetric ($\beta_{ij}=\beta_{ji}$) and as a consequence the metric is potential in canonical coordinates (i.e. $H_i^2=\d_i\varphi$ for some function $\varphi$). Moreover it is easy to check that the rotation coefficients satisfy the following overdetermined system of PDEs:
\begin{align}
\label{ED1}
\d_k\beta_{ij}=&\beta_{ik}\beta_{kj}, && i\ne j\ne k\ne i,\\
\label{ED2}
e(\beta_{ij})=&0, && i\ne j,\\
\label{ED3}
E(\beta_{ij})=&-\beta_{ij}, && i\ne j,
\end{align}  
where 
\[
e=\sum_{i=1}^n\d_i,\qquad E=\sum_{i=1}^nu^i\d_i.
\]
Condition \eqref{ED2} follows from \eqref{EqVIII}. The system (\ref{ED1},\ref{ED2}) is called \emph{Darboux-Egorov system} (see \cite{Darboux, Egorov}) and  implies the flatness of the metric $\eta$. The last condition \eqref{ED3}
 follows from the homogeneity properties. 
Given a solution of the above system, the Lam\'e coefficients $(H_1,...,H_n)$ are obtained by solving the overdetermined system of PDEs
\begin{align}
\label{ED5}
\d_jH_i=&\beta_{ij}H_j,&& i\ne j,\\
\label{ED6}
e(H_i)=&0, && \\
\label{ED7}
E(H_i)=&DH_i, && 
\end{align}
where $D=-\f{d}{2}$ is an eigenvalue of the skew-symmetric matrix $V_{ij}:=(u^j-u^i)\beta_{ij}$ \cite{D96}.
In dimension $n=3$, on the open set $u^1\neq u^2\neq u^3\neq u^1$, 
the general solution of the system (\ref{ED2}, \ref{ED3}) is
 \begin{equation}\label{relationsbetaF}
\begin{split}
\beta_{12}=\frac{1}{u^2-u^1}F_{12}\left(\frac{u^3-u^1}{u^2-u^1} \right)\,\,\,\\
\beta_{23}=\frac{1}{u^3-u^2}F_{23}\left(\frac{u^3-u^1}{u^2-u^1} \right)\,\,\,\\
\beta_{13}=\frac{1}{u^3-u^1}F_{13}\left(\frac{u^3-u^1}{u^2-u^1} \right).
\end{split}
\end{equation}
The remaining conditions \eqref{ED1} are equivalent to the following non-autonomous system of ODEs:
\begin{equation}\label{reductionODE3}
\begin{split}
  \frac{dF_{12}}{dz}&=\frac{1}{z(z-1)}F_{13}F_{23}\\
 \frac{dF_{13}}{dz}&=-\frac{1}{z-1}F_{12}F_{23}\\
 \frac{dF_{23}}{dz}&=\frac{1}{z}F_{12}F_{13}
\end{split}
\end{equation}
where $z:=\frac{u^3-u^1}{u^2-u^1}$.  It is well-known that three-dimensional Dubrovin-Frobenius manifolds are parameterized by solutions of a family of Painlev\'e VI equation (see \cite{D96}). This can be easily proved also studying system \eqref{reductionODE3}.

\begin{theorem}\label{painleveVI}
System \eqref{reductionODE3} is equivalent to the following  sigma form of Painlev\'e VI equation:
\begin{equation}\label{sigmapainleve0}
z^2(z-1)^2(\sigma'')^2+4\left[\sigma'\left(z\sigma'-\sigma\right)^2-(\sigma')^2(z\sigma'-\sigma) \right]=
-2R^2(\sigma')^2+R^4\sigma', 
\end{equation}
where the parameter $R^2$ is the value of the first integral $I=F_{12}^2+F_{13}^2+F_{23}^2$. 
\end{theorem}
\begin{proof}
First notice that $\frac{dI}{dz}=0$ as a simple computation shows. So we set $I=R^2$. Following \cite{AL13} it is easy to check that one can write
 the squares of the functions $F_{ij}$ in terms of a single function $\sigma(z)$:
\begin{eqnarray}
\label{position12}
F_{12}^2&=&\sigma',\\
\label{position13}
F_{13}^2&=&\sigma-z\sigma'+\f{R^2}{2},\\
\label{position23}
F_{23}^2&=&-\sigma+(z-1)\sigma'+\f{R^2}{2}.
\end{eqnarray}
From equations \eqref{position12}, \eqref{position13} and \eqref{position23} we have immediately 
\begin{equation}\label{step1}
z\frac{d}{dz}\left(F_{23}^2\right)=z(z-1)\frac{d}{dz}\left(F_{12}^2\right)=-(z-1)\frac{d}{dz}\left( F_{13}^2\right)=z(z-1)\sigma''(z).
\end{equation}
On the other hand, due to \eqref{reductionODE3} we have 
$$z\frac{d}{dz}\left(F_{23}^2\right)=z(z-1)\frac{d}{dz}\left(F_{12}^2\right)=-(z-1)\frac{d}{dz}\left( F_{13}^2\right)=2F_{12}F_{13}F_{32}.$$
By comparing these equations with \eqref{step1} and taking the square we obtain \eqref{sigmapainleve0}.
 \end{proof}
In dimension 4 there is a special class of Dubrovin-Frobenius manifolds that are also related to Painlev\'e VI equation \cite{Romano}. Dropping the assumption
 of symmetry of the rotation coefficients and allowing different degrees of homogeneity for the Lam\'e coefficients one ends up with the Darboux-Egorov system
 (\ref{ED1},\ref{ED2}) with the additional constraint 
\begin{equation}
\label{ED3bis}
E(\beta_{ij})=(d_i-d_j-1)\beta_{ij},\qquad i\ne j.
\end{equation}
In dimension $3$ the system $(\ref{ED1},\ref{ED2},\ref{ED3bis})$ reduces to a system of $6$ ODEs that turned out to be equivalent to the full family of Painlev\'e VI \cite{Limrn}. The corresponding geometric structure is a generalization of Dubrovin-Frobenius manifold structure and it is called bi-flat structure \cite{AL13}. 
 A similar result can be obtained studying the system (see \cite{ALmulti})
\begin{eqnarray}
\label{BF1}
&&\d_k\Gamma^i_{ij}=-\Gamma^i_{ij}\Gamma^i_{ik}+\Gamma^i_{ij}\Gamma^j_{jk}
+\Gamma^i_{ik}\Gamma^k_{kj}, \quad i\ne k\ne j\ne i,\\
\label{BF2}
&&e(\Gamma^i_{ij})=0,\qquad i\ne j\\
\label{BF3}
&&E(\Gamma^i_{ij})=-\Gamma^i_{ij},\qquad i\ne  j.
\end{eqnarray}
System \eqref{BF1} is called Darboux-Tsarev system. Regular non semisimple bi-flat structures in dimension 3 are also related to Painlev\'e transcendents. This was proved in \cite{ALmulti} studying the analogue of the Darboux-Tsarev system in the non semisimple case (see also \cite{KM} for an alternative approach based on the study of Okubo-type systems). 
 
	\section{Dubrovin-Frobenius metric in the general regular case}
	Let $M$ be a non semisimple Dubrovin-Frobenius manifold of dimension $n$, with commutative and associative product $\circ$, metric $\eta$, unit vector field $e$ and Euler vector field $E$. Let $M$ be \emph{regular} near a point $m\in M$, meaning that each Jordan block of the operator $L=E\,\circ$ is associated to a different eigenvalue.
	
	Let $r$ be the number of Jordan blocks of $L$ and let $m_{1},\dots,m_{r}$ be their sizes. Any set of coordinates $u^1,\dots,u^n$ for $M$ can be re-labelled by means of the following notation: for each $\alpha\in\{2,\dots,r\}$ and for each $j\in\{1,\dots,m_\alpha\}$ we write
	\begin{equation}
		\label{relabellingcoordinates}
		j(\alpha)=m_1+\dots+m_{\alpha-1}+j
	\end{equation}
	(for $\alpha=1$ we set $j(\alpha)=j$) so that $u^{j(\alpha)}$ denotes the $j$-th coordinate associated to the $\alpha$-th Jordan block. From now on, we will write $u^i$ when seeing the coordinate as running from $1$ to the dimension of the manifold and we will write $u^{i(\alpha)}$ when in need to highlight the Jordan block to which the coordinate refers. According to this notation, $\partial_i$ and $\partial_{i(\alpha)}$ will denote the partial derivative with respect to $u^i$ and $u^{i(\alpha)}$ respectively.
	
	In \cite{DH} David and Hertling provide a generalization of canonical coordinates in the regular case. According to their results we can assume that the product has the following form:
	\begin{equation}
		\label{CCanDH}
		\partial_{i(\alpha)}\circ\partial_{j(\beta)}=
		\begin{cases}
			\delta_{\alpha\beta}\,\partial_{(i+j-1)(\alpha)}\qquad&i+j\leq m_\alpha+1\\
			0&i+j\geq m_\alpha+2
		\end{cases}
	\end{equation}
	for all $i\in\{1,\dots,m_\alpha\}$, $j\in\{1,\dots,m_\beta\}$ for each $\alpha,\beta\in\{1,\dots,r\}$. The unit vector field takes the form
	\begin{equation}
		\label{unitDH}
		e=\overset{r}{\underset{\alpha=1}{\sum}}\,\partial_{1(\alpha)}
	\end{equation}
	and the Euler vector field becomes
	\begin{equation}
		\label{EulerDH}
		E=\overset{n}{\underset{s=1}{\sum}}\,u^s\,\partial_s\,.
	\end{equation}
	The operator $L=E\,\circ$ is given by
		\begin{equation}
			\label{Lallblocks}
			L=L^{i(\alpha)}_{j(\beta)}\,\partial_{i(\alpha)}\otimes du^{j(\beta)}
		\end{equation}
		where
		\begin{equation}
			\label{Lallblocks_cpt}
			L^{i(\alpha)}_{j(\beta)}=\begin{cases}
				\delta_{\alpha\beta}\,u^{(i-j+1)(\alpha)}\qquad&i\geq j\\
				0&i<j
			\end{cases}
		\end{equation}
		for $\alpha,\beta\in\{1,\dots,r\}$ and $i\in\{1,\dots,m_\alpha\}$, $j\in\{1,\dots,m_\beta\}$.
	
       In fact, given $\alpha,\beta\in\{1,\dots,r\}$ and $i\in\{1,\dots,m_\alpha\}$, $j\in\{1,\dots,m_\beta\}$ we have
		\begin{align}
			L^{i(\alpha)}_{j(\beta)}&=\big(E\,\circ\,\partial_{j(\beta)}\big)^{i(\alpha)}=u^{k(\gamma)}\,\big(\partial_{k(\gamma)}\,\circ\,\partial_{j(\beta)}\big)^{i(\alpha)}\notag\\&=
			\begin{cases}
				u^{k(\gamma)}\,\delta_{\beta\gamma}\,\big(\partial_{(j+k-1)(\beta)}\big)^{i(\alpha)}\qquad&1\leq k\leq m_\beta-j+1\\
				0&\textnormal{otherwise}
			\end{cases}\notag\\&=
			\begin{cases}
				u^{k(\beta)}\,\delta_{\alpha\beta}\,\delta^i_{j+k-1}\qquad\qquad\qquad&1\leq i-j+1\\
				0&\textnormal{otherwise}
			\end{cases}\notag\\&=
			\begin{cases}
				\delta_{\alpha\beta}\,u^{(i-j+1)(\alpha)}\qquad\qquad\qquad\,\,\,\,&i\geq j\\
				0&i<j.
			\end{cases}\notag
		\end{align}
	\noindent
	In order for the data $(\eta,\circ,e,E)$ to define an actual Dubrovin-Frobenius manifold, we have to impose all the axioms entering its definition.

	\indent
	In particular, we want to study conditions \eqref{EqI}--\eqref{EqVIII} in David-Hertling canonical coordinates. As stated in \cite{DH}, the metric $\eta$ is represented by a matrix whose entries vanish when corresponding to different Jordan blocks of $L$ and each block corresponds to an upper triangular Hankel submatrix. This follows from \eqref{EqI}. Precisely
	\begin{equation}
		\label{metricDHtensor}
		\eta=\delta_{\alpha\beta}\,\overline{\eta}_{(i+j-1)(\alpha)}\,du^{i(\alpha)}\otimes du^{j(\beta)}
	\end{equation}
	for some functions $\big\{\overline{\eta}_{(i)(\alpha)}\,|\,1\leq\alpha\leq r,\,1\leq i\leq m_\alpha\big\}$ and $\overline{\eta}_{(i)(\alpha)}=0$ for $i\geq m_\alpha+1$. Moreover, \eqref{EqIII} implies the existence of a \emph{metric potential} $H$ such that
	\begin{equation}
		\label{potentialDHmetric}
		\overline{\eta}_{i(\alpha)}=\partial_{i(\alpha)}H
	\end{equation}
	for all $i\in\{1,\dots,m_\alpha\}$ for each $\alpha\in\{1,\dots,r\}$.
	\bigskip\\
	\noindent
	Since we consider non semisimple Dubrovin-Frobenius manifolds, there must exist at least one Jordan block of size greater or equal than $2$. Without loss of generality we then assume that the size of the first Jordan block is greater than $1$. If one drops this assumption, analogous results will hold, where different coordinates will play the roles here played by $u^1$, $u^2$.
		
	\indent
	If we take into account that the metric must be homogeneous with respect to the Euler vector field and constant with respect to the unity vector field, we are able to get a further expression for the terms $\overline{\eta}_{(i)(\alpha)}$s.
	\begin{thm}
		The functions $\overline{\eta}_i$s appearing in \eqref{metricDHtensor} can be written as
		\begin{equation}
			\label{metricconj}
			\overline{\eta}_i=(u^2)^{-d}\,F_i\qquad\qquad i\in\{1,\dots,n\}
		\end{equation}
		for some functions $F_1,\dots,F_n$ of the variables
		\begin{equation}
			\label{zjconj}
			z^j=\frac{u^{j+2}-u^1\,\overset{r}{\underset{\alpha=2}{\sum}}\,\delta^{j+2}_{1(\alpha)}}{u^2}\qquad\qquad j\in\{1,\dots,n-2\}
		\end{equation}
		such that
		\begin{align}
			&F_1=-\overset{r}{\underset{\alpha=2}{\sum}}\,\partial_{z^{1(\alpha)-2}}f+C_1\label{F1conj}\\
			&F_2=-z^j\,\partial_{z^j}f-(d-1)\,f+C_2\label{F2conj}\\
			&F_j=\partial_{z^{j-2}}f\qquad\qquad\qquad\qquad\qquad\qquad j\in\{3,\dots,n\}\label{Fjconj}
		\end{align}
		for some function $f$ of $z^1,\dots,z^{n-2}$ and constants $C_1$, $C_2$. In particular, the quantity
		\begin{equation}
			\label{integraleprimo1conj}
			\overset{r}{\underset{\alpha=1}{\sum}}\,F_{1(\alpha)}=C_1
		\end{equation}
		is a constant that vanishes whenever $d\neq0$.
	\end{thm}
	\begin{proof}
		By imposing \eqref{EqVIII} we get
		\begin{equation}
			\notag
			\overset{r}{\underset{\alpha=1}{\sum}}\,\partial_{1(\alpha)}\overline{\eta}_i=\mathcal{L}_e\overline{\eta}_i=0
		\end{equation}
		for $i\in\{1,\dots,n\}$. It follows that each $\overline{\eta}_i$ can be written as
		\begin{equation}
			\overline{\eta}_i(u^1,\dots,u^n)=\varphi_i\bigg(u^2,u^3-u^1\,\overset{r}{\underset{\alpha=2}{\sum}}\,\delta^{3}_{1(\alpha)},\dots,u^n-u^1\,\overset{r}{\underset{\alpha=2}{\sum}}\,\delta^{n}_{1(\alpha)}\bigg)
		\end{equation}
		for some function $\varphi_i$ of $n-1$ variables. By the homogeneity condition \eqref{homog}, it can be rewritten as in \eqref{metricconj} for some function $F_i$ of the variables defined in \eqref{zjconj}.
		
		The flatness of $e$ with respect to $\nabla$ implies that $d\big(\eta(e,\cdot)\big)=0$ (see \cite{DH}), that is
		\begin{equation}
			\notag
			\partial_{j(\beta)}\overline{\eta}_{i(\alpha)}\,du^{j(\beta)}\wedge du^{i(\alpha)}=0
		\end{equation}
		thus
		\begin{equation}
			\label{starconj}
			\partial_{j(\beta)}\overline{\eta}_{i(\alpha)}-\partial_{i(\alpha)}\overline{\eta}_{j(\beta)}=0
		\end{equation}
		for all $i\in\{1,\dots,m_\alpha\}$, $j\in\{1,\dots,m_\beta\}$ and $\alpha,\beta\in\{1,\dots,r\}$. In particular, for $i(\alpha),j(\beta)\in\{3,\dots,n\}$ we get
		\begin{equation}
			\notag
			\partial_{z^{j(\beta)-2}}F_{i(\alpha)}=\partial_{z^{i(\alpha)-2}}F_{j(\beta)}.
		\end{equation}
		There must then exist a function $f$ of the variables $z^1,\dots,z^{n-2}$ realizing \eqref{Fjconj}. By fixing $j(\beta)=2$ and $i(\alpha)\in\{3,\dots,n\}$ in \eqref{starconj} we obtain the following relation:
		\begin{equation}
			\notag
			\partial_{i(\alpha)}\big((u^2)^{-d}\,F_2\big)=\partial_2\big((u^2)^{-d}\,F_{i(\alpha)}\big)
		\end{equation}
		which amounts to
		\begin{equation}
			\notag
			(u^2)^{-d-1}\,\partial_{z^{i(\alpha)-2}}F_2=-d\,(u^2)^{-d-1}\,F_{i(\alpha)}+(u^2)^{-d}\,\partial_2F_{i(\alpha)}
		\end{equation}
		and
		\begin{equation}
			\notag
			\partial_{z^{i(\alpha)-2}}F_2=-d\,F_{i(\alpha)}-\overset{n-2}{\underset{j=1}{\sum}}\,z^j\,\partial_{z^j}F_{i(\alpha)}.
		\end{equation}
		By taking into account \eqref{Fjconj} we get
		\begin{equation}
			\notag
			\partial_{z^{i(\alpha)-2}}F_2=-d\,\partial_{z^{i(\alpha)-2}}f-\overset{n-2}{\underset{j=1}{\sum}}\,z^j\,\partial_{z^j}\partial_{z^{i(\alpha)-2}}f.
		\end{equation}
		Then for each $i\in\{1,\dots,n-2\}$
		\begin{equation}
			\notag
			\partial_{z^{i}}F_2=-d\,\partial_{z^{i}}f-z^j\,\partial_{z^j}\partial_{z^{i}}f
		\end{equation}
		that is
		\begin{equation}
			\notag
			\partial_{z^{i}}\bigg[F_2-(1-d)\,f+z^j\,\partial_{z^j}f\bigg]=0.
		\end{equation}
		Therefore the quantity $F_2-(1-d)\,f+z^j\,\partial_{z^j}f$ equals some constant $C_2$, proving \eqref{F2conj}.
		
		By taking $i(\alpha)=1(\alpha)$, $j(\beta)\in\{3,\dots,n\}$ in \eqref{starconj} and summing over all $\alpha\in\{1,\dots,r\}$ we get
		\begin{equation}
			\notag
			\overset{r}{\underset{\alpha=1}{\sum}}\,\partial_{j(\beta)}\overline{\eta}_{1(\alpha)}=\overset{r}{\underset{\alpha=1}{\sum}}\,\partial_{1(\alpha)}\overline{\eta}_{j(\beta)}
		\end{equation}
		that is
		\begin{equation}
			\notag
			(u^2)^{-d}\,\partial_{j(\beta)}\Bigg(\,\overset{r}{\underset{\alpha=1}{\sum}}\,F_{1(\alpha)}\Bigg)=\mathcal{L}_e\overline{\eta}_{j(\beta)}
		\end{equation}
		thus
		\begin{equation}
			\notag
			\partial_{z^{j(\beta)-2}}\Bigg(\,\overset{r}{\underset{\alpha=1}{\sum}}\,F_{1(\alpha)}\Bigg)=0.
		\end{equation}
		This means that
		\begin{equation}
			\notag
			\partial_{z^{j}}\Bigg(\,\overset{r}{\underset{\alpha=1}{\sum}}\,F_{1(\alpha)}\Bigg)=0
		\end{equation}
		for all $j\in\{1,\dots,n-2\}$, proving that $\overset{r}{\underset{\alpha=1}{\sum}}\,F_{1(\alpha)}$ must be be equal to some constant $C_1$. Condition \eqref{F1conj} follows.
		
		On the other hand, by taking $i(\alpha)=1(\alpha)$, $j(\beta)=2$ in \eqref{starconj} and summing over all $\alpha\in\{1,\dots,r\}$ we get
		\begin{equation}
			\notag
			\partial_2\Bigg(\,\overset{r}{\underset{\alpha=1}{\sum}}\,(u^2)^{-d}\,F_{1(\alpha)}\Bigg)=0
		\end{equation}
		which, since $\overset{r}{\underset{\alpha=1}{\sum}}\,F_{1(\alpha)}=C_1$, amounts to
			\begin{equation}
				\notag
				\partial_2\big(\,(u^2)^{-d}\,C_1\big)=0.
			\end{equation}
			This implies $d\,C_1=0$, meaning that the constant $C_1$ must vanish whenever $d\neq0$.		
	\end{proof}
	\begin{prop}
		Up to constants, the function $f$ appearing in \eqref{F1conj}, \eqref{F2conj}, \eqref{Fjconj} is related to the metric potential $H$ by the following formula:
			\begin{equation}
				\label{fHrelazione}
				H=(u^2)^{1-d}\,f+C_2\,\varphi(u^2)+C_1\,u^1
			\end{equation}
			where
			\begin{equation}
				\varphi(u^2)=\begin{cases}
					\frac{(u^2)^{1-d}}{1-d}\qquad&\text{if }d\neq1\\
					\ln{u^2}\qquad&\text{if }d=1.
				\end{cases}
		\end{equation}
	\end{prop}
	\begin{proof}
		By \eqref{potentialDHmetric} and \eqref{metricconj} we have
		\begin{equation}
			\label{Hfstart}
			\partial_iH=(u^2)^{-d}\,F_i\big(z^1,\dots,z^{n-2}\big)
		\end{equation}
		for each $i\in\{1,\dots,n\}$. For $i\geq3$ we get
		\begin{equation}
			\notag
			\partial_iH=(u^2)^{-d}\,\partial_{z^{i-2}}f
		\end{equation}
		that is
		\begin{equation}
			\notag
			\partial_{z^{i-2}}H=(u^2)^{1-d}\,\partial_{z^{i-2}}f
		\end{equation}
		or		
		\begin{equation}
			\notag
			\partial_{z^{i-2}}\big(H-(u^2)^{1-d}\,f\big)=0.
		\end{equation}
		It follows that
		\begin{equation}
			\label{almostH}
			H=(u^2)^{1-d}\,f+K(u^1,u^2)
		\end{equation}
		for some function $K(u^1,u^2)$. For $i=2$ in \eqref{Hfstart} we get
		\begin{equation}
			\notag
			\partial_2H=(u^2)^{-d}\,\big(-z^j\,\partial_{z^j}f-(d-1)\,f+C_2\big)
		\end{equation}
		that is
		\begin{equation}
			\notag
			(u^2)^{-d}\big((1-d)\,f-z^{j}\,\partial_{z^j}f\big)+\partial_2K=(u^2)^{-d}\,\big(-z^j\,\partial_{z^j}f-(d-1)\,f+C_2\big)
		\end{equation}
		yielding
		\begin{equation}
			\notag
			\partial_2K(u^1,u^2)=C_2\,(u^2)^{-d}.
		\end{equation}
		Then
			\begin{equation}
				\label{almostK}
				K(u^1,u^2)=\begin{cases}
					C_2\,\frac{(u^2)^{1-d}}{1-d}+k(u^1)\qquad&\text{if }d\neq1\\
					C_2\,\ln{u^2}+k(u^1)\qquad&\text{if }d=1.
				\end{cases}
			\end{equation}
			for some function $k(u^1)$. By putting together \eqref{almostH} and \eqref{almostK} one gets
			\begin{equation}
				\notag
				H=(u^2)^{1-d}\,f+C_2\,\varphi(u^2)+k(u^1)
			\end{equation}
			for
			\begin{equation}
				\label{varphiu2}
				\varphi(u^2)=\begin{cases}
					\frac{(u^2)^{1-d}}{1-d}\qquad&\text{if }d\neq1\\
					\ln{u^2}\qquad&\text{if }d=1.
				\end{cases}
			\end{equation}
			For $i=1$ in \eqref{Hfstart} we finally get
			\begin{equation}
				\notag
				\partial_1H=(u^2)^{-d}\,F_1
			\end{equation}
			that is
			\begin{equation}
				\notag
				-(u^2)^{-d}\,\overset{r}{\underset{\alpha=2}{\sum}}\,\partial_{z^{1(\alpha)-2}}f+\partial_1k(u^1)=-(u^2)^{-d}\,\overset{r}{\underset{\alpha=2}{\sum}}\,\partial_{z^{1(\alpha)-2}}f+(u^2)^{-d}\,C_1.
			\end{equation}
			Thus
			\begin{equation}
				\notag
				\partial_1k(u^1)=(u^2)^{-d}\,C_1=\begin{cases}
					0\qquad&\text{if }d\neq0\\
					C_1\qquad&\text{if }d=0
				\end{cases}=C_1
			\end{equation}
			implying
			\begin{equation}
				\notag
				k(u^1)=C_1\,u^1+C_3
			\end{equation}
			for some constant $C_3$. We conclude that
			\begin{equation}
				\notag
				H=(u^2)^{1-d}\,f+C_2\,\varphi(u^2)+C_1\,u^1+C_3
			\end{equation}
			for $\varphi(u^2)$ as in \eqref{varphiu2}.
	\end{proof}
	\bigskip
	\noindent

	\section{The case of a single Jordan block: explicit results up to dimension $4$}
	In this section we classify regular non semisimple Dubrovin-Frobenius manifold structures up to dimension $4$ in the case where the operator $L$ has a single Jordan block.
 Due to the results of the previous Section in the specific case where $L$ has a single Jordan block of size $n$ the unit vector field becomes $e=\partial_1$ and in canonical coordinates we have
	\begin{equation}
		\label{CCanDH_nD1J}
		\partial_{i}\circ\partial_{j}=
		\begin{cases}
			\partial_{i+j-1}\qquad&i+j\leq n+1\\
			0&i+j\geq n+2
		\end{cases}
	\end{equation}
	for all $i,j\in\{1,\dots,n\}$ and $u^i=u^{i(1)}$ for each $i\in\{1,\dots,n\}$. The operator $L$ is described by the following lower triangular Toeplitz matrix:
	\begin{equation}
		L=\begin{bmatrix}u^1&0&0&\dots&0&0\\u^2&u^1&0&\dots&0&0\\u^3&u^2&u^1&\dots&0&0\\\vdots&\vdots&\vdots&\ddots&\vdots&\vdots\\u^{n-1}&u^{n-2}&u^{n-3}&\dots&u^1&0\\u^n&u^{n-1}&u^{n-2}&\dots&u^2&u^1
		\end{bmatrix}.
	\end{equation}
	The metric is represented by an upper triangular Hankel matrix that only depends on the coordinate $u^2$ and on $n$ functions $F_1,\dots,F_n$ of the variables
	\begin{equation}
		\notag
		z^i=\frac{u^{i+2}}{u^2}\qquad i\in\{1,\dots,n-2\}.
	\end{equation}
	It takes the following form:
	\begin{align}
		\label{mymetric}
		\eta=(u^2)^{-d}\,\begin{bmatrix}
			F_1&F_2&F_3&\dots&F_{n-1}&F_n\\F_2&F_3&F_4&\dots&F_n&0\\F_3&F_4&F_5&\dots&0&0\\\vdots&\vdots&\vdots&\ddots&\vdots&\vdots\\F_{n-1}&F_n&0&\dots&0&0\\F_n&0&0&\dots&0&0
		\end{bmatrix}.
	\end{align}
	In particular, $F_1$ is equal to a constant $C_1$ that vanishes whenever $d\neq0$ and the other $F_i$s are expressed in terms of a function $f(z^1,\dots,z^{n-2})$ by
	\begin{align}
		&F_2=-z^i\,\partial_{z^i}f-(d-1)\,f+C_2\label{F2}\\
		&F_j=\partial_{z^{j-2}}f\label{Fj}\qquad\forall j\in\{3,\dots,n\}
	\end{align}
	for some constant $C_2$.\indent

	\subsection{Dimension $n=2$}
	Let $M$ be a two-dimensional Dubrovin-Frobenius manifold with product $\circ$, metric $\eta$, unit vector field $e$ and Euler vector field $E$. Let us require $M$ to be regular and the operator $L=E\,\circ$ to have a single Jordan block near a point $m\in M$. The unit and the Euler vector fields read respectively $e=\partial_1$ and $E=u^1\partial_1+u^2\partial_2$. It follows directly from \eqref{mymetric} that the metric has the form
	\begin{align}
		\label{metric2}
		\eta=(u^2)^{-d}\begin{bmatrix}
			C_1&C_2\\C_2&0
		\end{bmatrix}
	\end{align}
	for some constant $C_1$ which vanishes whenever $d\neq0$ and for some non-zero constant $C_2$.
	\bigskip
	\\We are able to recover flat coordinates and an explicit expression for the Dubrovin-Frobenius prepotential, as pointed out in the following result.
	\begin{thm}
		Flat coordinates coincide with the canonical ones when $d=0$. Otherwise, they are given by
		\begin{align}
			&x^1(u^1,u^2)=u^1\notag\\
			&x^2(u^1,u^2)=\frac{(u^2)^{1-d}}{1-d}\notag
		\end{align}
		when $d\neq1$ and by
		\begin{align}
			&x^1(u^1,u^2)=u^1\notag\\
			&x^2(u^1,u^2)=\ln{u^2}\notag
		\end{align}
		when $d=1$. In all the cases, the prepotential is given by
		\begin{align}
			\label{Fpot2}
			F(x^1,x^2)=\frac{C_1}{6}\,(x^1)^3+\frac{C_2}{2}\,(x^1)^2\,x^2
		\end{align}
		up to second-order polynomial terms. In flat coordinates the unit and the Euler vector fields are respectively given by $\overset{\sim}{e}=\overset{\sim}{\partial_1}$ and
		\begin{equation}
			\notag
			\overset{\sim}{E}=\begin{cases}
				x^1\overset{\sim}{\partial_1}+\overset{\sim}{\partial_2}\qquad\qquad\qquad&\text{if }d=1\\
				x^1\overset{\sim}{\partial_1}+x^2(1-d)\overset{\sim}{\partial_2}\qquad&\text{if }d\neq1.
			\end{cases}
		\end{equation}
	\end{thm}	
	\begin{proof}
		If $d=0$ then the metric in \eqref{metric2} is constant, thus flat coordinates coincide with the canonical ones. Let us now fix $d\neq0$. In 
		this case the flat coordinates are
		\begin{align}
			&x^1(u^1,u^2)=u^1\notag\\
			&x^2(u^1,u^2)=\frac{(u^2)^{1-d}}{1-d}\notag
		\end{align}
		when $d\neq1$ and
		\begin{align}
			&x^1(u^1,u^2)=u^1\notag\\
			&x^2(u^1,u^2)=\ln{u^2}\notag
		\end{align}
		when $d=1$. In both cases, in flat coordinates the metric becomes\[
		\tilde{\eta}=\begin{bmatrix}
			C_1&C_2\\C_2&0
		\end{bmatrix}
		\]and the structure constants equal the ones in canonical coordinates:
		\[\tilde{c}^k_{ij}=c^k_{ij}\qquad i,j,k\in\{1,2\}.\]
		It follows that up to second-order polynomial terms the Dubrovin-Frobenius prepotential $F$ is of the form
		\begin{align}
			F(x^1,x^2)=&\,\frac{C_1}{6}(x^1)^3+\frac{C_2}{2}(x^1)^2\,x^2\notag
		\end{align}
		and that in flat coordinates the unit and the Euler vector fields become of the form stated above.
	\end{proof}
	
	\subsection{Dimension $n=3$}
	Let $M$ be a three-dimensional Dubrovin-Frobenius manifold with product $\circ$, metric $\eta$, unit vector field $e$ and Euler vector field $E$. Let us require $M$ to be regular and the operator $L=E\,\circ$ to have a single Jordan block near a point $m\in M$. The unit and the Euler vector fields read respectively $e=\partial_1$ and $E=u^1\partial_1+u^2\partial_2+u^3\partial_3$. We already know from \eqref{mymetric} that the metric is of the form
	\begin{align}
		\label{metric3}
		\eta=(u^2)^{-d}\,\begin{bmatrix}
			F_1(\frac{u^3}{u^2})&F_2(\frac{u^3}{u^2})&F_3(\frac{u^3}{u^2})\\F_2(\frac{u^3}{u^2})&F_3(\frac{u^3}{u^2})&0\\F_3(\frac{u^3}{u^2})&0&0
		\end{bmatrix}
	\end{align}
	for some functions $F_1$, $F_2$, $F_3$ and that $F_1$ is equal to a constant $C_1$ that vanishes whenever $d\neq0$. It turns out from the zero-curvature conditions that the functions $F_2$, $F_3$ must be solutions to the following  system of ODEs 
		\begin{equation}
			\label{3.1_ODEs}
			\begin{cases}
				F_2'+z\,F_3'+d\,F_3=0\\
				2\,F_3\,F_3''-3\,(F_3')^2=0.
			\end{cases}
		\end{equation}
 In fact, let us introduce the variable $z=\frac{u^3}{u^2}$. We have already seen that there exists a function $f(z)$ such that
		\begin{align}
			&F_2(z)=-z\,f'(z)-(d-1)\,f(z)+C_2\notag\\
			&F_3(z)=f'(z)\notag
		\end{align}
		for some constant $C_2$. It follows that
		\begin{equation}
			\notag
			F_2'+z\,F_3'+d\,F_3=0.
		\end{equation}
		Moreover, by requiring that $R^1_{232}=0$ one obtains the Liouville-type differential equation
		\begin{equation}
			\notag
			2\,F_3\,F_3''-3\,(F_3')^2=0.
		\end{equation}
		This suffices to make all of the conditions in \eqref{EqI}, \eqref{EqII}, \eqref{EqIII}, \eqref{homog}, \eqref{EqVIII}, \eqref{EqXI} hold without imposing more. So far, what we know about the functions $F_1$, $F_2$, $F_3$ is that $F_1$ equals some constant $C_1$ and that $F_2$, $F_3$ are solutions to the system \eqref{3.1_ODEs}. Two expressions for the function $f$ appering in \eqref{F2}, \eqref{Fj} are then possible, as shown below.
	\begin{thm}
		The function $f$ realizing \eqref{F2}, \eqref{Fj} is either provided by
		\begin{equation}
			\label{3.1fcaso1}
			f(z)=C_3\,z+C_4
		\end{equation}
		for some constants $C_3$, $C_4$ or by
		\begin{equation}
			\label{3.1fcaso2}
				f(z)=\frac{C_4}{z+C_3}+C_5
		\end{equation}
		for some constants $C_3$, $C_4$, $C_5$.
	\end{thm}
	\begin{proof}
		The first condition in \eqref{3.1_ODEs} amounts to \eqref{F2} and \eqref{Fj}, while the second one can be rewritten as
		\begin{equation}
			\label{solvef3}
			2\,f'(z)\,f'''(z)-3\,(f''(z))^2=0.
		\end{equation}
		Assuming $f''(z)\ne 0$ the solutions to equation \eqref{solvef3} can be  written as \eqref{3.1fcaso2}, while \eqref{3.1fcaso1} is recovered by considering solutions corresponding to $f''(z)=0$.
	\end{proof}
 Summarizing two cases may occur: either
	\begin{equation}
		\label{2_caso3}
		\begin{cases}
			&F_1(z)=C_1\\
			&F_2(z)=-C_3\,d\,z+C_2\\
			&F_3(z)=C_3
		\end{cases}
	\end{equation}
	for some constant $C_1$ that vanishes for $d\neq0$ and some constants $C_2$, $C_3$ or
	\begin{equation}
		\label{2_caso4}
		\begin{cases}
			&F_1(z)=C_1\\
			&F_2(z)=\frac{C_3\,C_4}{(z+C_3)^2}-\frac{(2-d)\,C_4}{z+C_3}+C_2\\
			&F_3(z)=\frac{C_4}{(z+C_3)^2}
		\end{cases}
	\end{equation}
	for some constant $C_1$ that vanishes for $d\neq0$ and some constants $C_2$, $C_3$, $C_4$.
	\begin{prop}
		In the case of \eqref{2_caso3} flat coordinates are given by
		\begin{align}
			&x^1(u^1,u^2,u^3)=u^1\notag\\
			&x^2(u^1,u^2,u^3)=(u^2)^{-d}\,u^3+\frac{C_2(u^2)^{1-d}}{C_3(1-d)}\notag\\
			&x^3(u^1,u^2,u^3)=\frac{2}{2-d}(u^2)^{\frac{2-d}{2}}\notag
		\end{align}
		when $d\notin\{0,1,2\}$, by
		\begin{align}
			&x^1(u^1,u^2,u^3)=u^1\notag\\
			&x^2(u^1,u^2,u^3)=\frac{u^3}{(u^2)^2}-\frac{C_2}{C_3\,u^2}\notag\\
			&x^3(u^1,u^2,u^3)=\ln{u^2}\notag
		\end{align}
		when $d=2$, by
		\begin{align}
			&x^1(u^1,u^2,u^3)=u^1\notag\\
			&x^2(u^1,u^2,u^3)=\frac{u^3}{u^2}+\frac{C_2}{C_3}\,\ln{u^2}\notag\\
			&x^3(u^1,u^2,u^3)=2\,\sqrt{u^2}\notag
		\end{align}
		when $d=1$ and, trivially, by
		\begin{align}
			&x^1(u^1,u^2,u^3)=u^1\notag\\
			&x^2(u^1,u^2,u^3)=u^2\notag\\
			&x^3(u^1,u^2,u^3)=u^3\notag
		\end{align}
		when $d=0$.
	\end{prop}
	The proof is a straightforward computation.
	
	\begin{prop}
		Let $x^1,\,x^2,\,x^3$ denote flat coordinates. Up to second-order polynomial terms, in the case of \eqref{2_caso3} the prepotential is given by
		\begin{equation}
			F(x^1,x^2,x^3)=\frac{C_3}{2}\,(x^1)^2\,x^2+\frac{C_3}{2}\,x^1\,(x^3)^2
		\end{equation}
		when $d\neq0$ and by
		\begin{equation}
			F(x^1,x^2,x^3)=\frac{C_1}{6}\,(x^1)^3+\frac{C_2}{2}\,(x^1)^2\,x^2+\frac{C_3}{2}\,(x^1)^2\,x^3+\frac{C_3}{2}\,x^1\,(x^2)^2
		\end{equation}
		when $d=0$.
	\end{prop}
	The proof is a straightforward computation.
\newline
\newline
Analogous results can be achieved for the case of \eqref{2_caso4}, as presented below.
	\begin{prop}
		In the case of \eqref{2_caso4} flat coordinates are given by
		\begin{align}
			x^1(u^1,u^2,u^3)&=\frac{C_1C_3\,u^1u^2+C_2C_3(u^2)^2+C_1\,u^1u^3+C_2\,u^2u^3-C_4(u^2)^2}{C_1\,(C_3\,u^2+u^3)}\notag\\
			x^2(u^1,u^2,u^3)&=\big[\big(-C_2C_3\sqrt{C_1C_4}-C_4(C_1-C_4+C_2C_3)\big)(u^2)^{\frac{2C_4-\sqrt{C_1C_4}}{C_4}}\notag\\&-C_2\,u^3(u^2)^{\frac{C_4-\sqrt{C_1C_4}}{C_4}}(C_4+\sqrt{C_1C_4})]\frac{1}{C_4(C_1-C_4)\,(C_3\,u^2+u^3)}\notag\\
			x^3(u^1,u^2,u^3)&=\big[\big(C_2C_3\sqrt{C_1C_4}-C_4(C_1-C_4+C_2C_3)\big)(u^2)^{\frac{2C_4+\sqrt{C_1C_4}}{C_4}}\notag\\&-C_2\,u^3(u^2)^{\frac{C_4+\sqrt{C_1C_4}}{C_4}}(C_4-\sqrt{C_1C_4})]\frac{1}{C_4(C_1-C_4)\,(C_3\,u^2+u^3)}\notag
		\end{align}
		when $d=0$, $C_1\neq0$ and $C_1\neq C_4$, by
		\begin{align}
			&x^1(u^1,u^2,u^3)=u^1+\frac{(u^2)^2(\ln{u^2})^2}{2\,(C_3\,u^2+u^3)}+\frac{C_2\,u^2\big(2\,\ln{u^2}-(\ln{u^2})^2-2\big)}{2\,C_4}\notag\\
			&x^2(u^1,u^2,u^3)=-\frac{(u^2)^2\,\ln{u^2}}{C_3\,u^2+u^3}+\frac{C_2\,u^2\big(\ln{u^2}-1\big)}{C_4}\notag\\
			&x^3(u^1,u^2,u^3)=-\frac{(u^2)^2}{C_3\,u^2+u^3}+\frac{C_2\,u^2}{C_4}\notag
		\end{align}
		when $d=0$ and $C_1=0$, by
		\begin{align}
			&x^1(u^1,u^2,u^3)=u^1-\frac{(u^2)^2}{C_3\,u^2+u^3}+\frac{C_2\,u^2}{C_4}\notag\\
			&x^2(u^1,u^2,u^3)=-\frac{(u^2)^3}{2\,\big(C_3\,u^2+u^3\big)}+\frac{C_2\,(u^2)^2}{4\,C_4}\notag\\
			&x^3(u^1,u^2,u^3)=-\frac{u^2}{C_3\,u^2+u^3}+\frac{C_2\,\ln{u^2}}{C_4}\notag
		\end{align}
		when $d=0$ and $C_1=C_4$, by
		\begin{align}
			x^1(u^1,u^2,u^3)&=\big[\big(C_4\,u^1(d)^2-2C_2\,u^2\big)u^2(C_3)^2+\big(u^1u^3(d)^2\notag\\&+2(u^2)^2\big)C_3C_4+2C_2(u^3)^2\big]\frac{1}{C_3C_4(d)^2(C_3\,u^2+u^3)}\notag\\
			x^2(u^1,u^2,u^3)&=\frac{\big(C_2C_3-C_4\,(1-d)\big)(u^2)^{2-d}+C_2\,u^3(u^2)^{1-d}}{C_4\,(1-d)\big(C_3\,u^2+u^3\big)}\notag\\
			x^3(u^1,u^2,u^3)&=\frac{2C_2\,u^3(u^2)^{\frac{2-d}{2}}-\big(C_4(2-d)-2C_2C_3\big)(u^2)^{\frac{4-d}{2}}}{C_4\,(2-d)\big(C_3\,u^2+u^3\big)}\notag
		\end{align}
		when $d\notin\{0,1,2\}$ and $C_3\neq0$, by
		\begin{align}
			x^1(u^1,u^2,u^3)&=\frac{\big(C_4\,(d)^2\,u^1-2C_2\,u^2\big)\,u^3+2C_4(u^2)^2}{C_4\,(d)^2\,u^3}\notag\\
			x^2(u^1,u^2,u^3)&=\frac{2C_2\,u^3(u^2)^{\frac{2-d}{2}}-C_4\,(2-d)(u^2)^{\frac{4-d}{2}}}{C_4\,(2-d)\,u^3}\notag\\
			x^3(u^1,u^2,u^3)&=\frac{C_2\,u^3(u^2)^{1-d}-C_4\,(1-d)(u^2)^{2-d}}{C_4\,(1-d)\,u^3}\notag
		\end{align}
		when $d\notin\{0,1,2\}$ and $C_3=0$, by
		\begin{align}
			&x^1(u^1,u^2,u^3)=u^1+\frac{2(u^2)^2}{C_3\,u^2+u^3}-\frac{2C_2\,u^2}{C_4}\notag\\
			&x^2(u^1,u^2,u^3)=-\frac{u^2}{C_3\,u^2+u^3}+\frac{C_2\,\ln{u^2}}{C_4}\notag\\
			&x^3(u^1,u^2,u^3)=-\frac{(u^2)^{\frac{3}{2}}}{C_3\,u^2+u^3}+\frac{2C_2\,\sqrt{u^2}}{C_4}\notag
		\end{align}
		when $d=1$, by
		\begin{align}
			&x^1(u^1,u^2,u^3)=u^1+\frac{(u^2)^2}{2\big(C_3\,u^2+u^3\big)}-\frac{C_2\,u^2}{2C_4}\notag\\
			&x^2(u^1,u^2,u^3)=-\frac{u^2}{C_3\,u^2+u^3}+\frac{C_2\,\ln{u^2}}{C_4}\notag\\
			&x^3(u^1,u^2,u^3)=-\frac{1}{C_3\,u^2+u^3}-\frac{C_2}{C_4\,u^2}\notag
		\end{align}
		when $d=2$.
	\end{prop}
	Here are explicit expressions for the Dubrovin-Frobenius potential in some selected cases.
	\begin{es}
		Let us fix $d=0$, $C_1=C_4$ and $C_2=0$. In flat coordinates the metric becomes
		\[
		\tilde{\eta}=\begin{bmatrix}
			C_4&0&0\\0&0&-C_4\\0&-C_4&0
		\end{bmatrix}
		\]
		and the prepotential is given by
		\begin{equation}
			F(x^1,x^2,x^3)=\frac{2\sqrt{2}}{3}C_4\,(x^2)^{\frac{3}{2}}(x^3)^{\frac{3}{2}}+\frac{C_4}{6}(x^1)^{3}-C_4\,x^1\,x^2\,x^3
		\end{equation}
		up to second-order polynomial terms. In flat coordinates the unit and the Euler vector fields are respectively written as
		\begin{equation}
			\notag
			\tilde{e}=\tilde{\partial}_1
		\end{equation}
		and
		\begin{equation}
			\notag
			\tilde{E}=x^1\,\tilde{\partial}_1+2\,x^2\,\tilde{\partial}_2.
		\end{equation}
	\end{es}
	\begin{es}
		Let us fix $d=2$ and $C_2=0$. In flat coordinates the metric becomes
		\[
		\tilde{\eta}=\begin{bmatrix}
			0&0&C_4\\0&C_4&0\\C_4&0&0
		\end{bmatrix}
		\]
		and the prepotential is given by
		\begin{equation}
			F(x^1,x^2,x^3)=\frac{C_4}{2}\,(x^1)^{2}\,x^3+\frac{C_4}{2}\,x^1\,(x^2)^2+\frac{C_4}{8}\,\frac{(x^2)^4}{x^3}
		\end{equation}
		up to second-order polynomial terms. In flat coordinates the unit and the Euler vector fields are respectively written as
		\begin{equation}
			\notag
			\tilde{e}=\tilde{\partial}_1
		\end{equation}
		and
		\begin{equation}
			\notag
			\tilde{E}=x^1\,\tilde{\partial}_1-x^3\,\tilde{\partial}_3.
		\end{equation}
	\end{es}
	\begin{es}
		Let us fix $d=2$ and $C_2=1$. In flat coordinates the metric becomes
		\[
		\tilde{\eta}=\begin{bmatrix}
			0&0&C_4\\0&C_4&0\\C_4&0&0
		\end{bmatrix}
		\]
		and the prepotential is given by
		\begin{align}
			F(x^1,x^2,x^3)&=\frac{1}{24(C_4)^3x^3}\,\bigg(3\,W\big(C_4\,x^3\,e^{C_4\,x^2-1}\big)^4+22\,W\big(C_4\,x^3\,e^{C_4\,x^2-1}\big)^3\\&+63\,W\big(C_4\,x^3\,e^{C_4\,x^2-1}\big)^2+72\,W\big(C_4\,x^3\,e^{C_4\,x^2-1}\big)\bigg)\notag\\&+\frac{C_4}{2}\,(x^1)^{2}\,x^3+\frac{C_4}{2}\,x^1\,(x^2)^2\notag
		\end{align}
		up to second-order polynomial terms, where $W$ denotes the principal branch of the Lambert \emph{W} function. In flat coordinates the unit and the Euler vector fields are respectively written as
		\begin{equation}
			\notag
			\tilde{e}=\tilde{\partial}_1
		\end{equation}
		and
		\begin{equation}
			\notag
			\tilde{E}=x^1\,\tilde{\partial}_1+\frac{1}{C_4}\,\tilde{\partial}_2-x^3\,\tilde{\partial}_3.
		\end{equation}
	\end{es}
	
	\subsection{Dimension $n=4$}
	Let $M$ be a four-dimensional Dubrovin-Frobenius manifold with product $\circ$, metric $\eta$, unit vector field $e$ and Euler vector field $E$. Let us require $M$ to be regular and the operator $L=E\,\circ$ to have a single Jordan block near a point $m\in M$. The unit and the Euler vector fields read respectively $e=\partial_1$ and $E=u^1\partial_1+u^2\partial_2+u^3\partial_3+u^4\partial_4$. We already know from \eqref{mymetric} that the metric is of the form
	\begin{align}
		\label{metric4}
		\eta=(u^2)^{-d}\,\begin{bmatrix}
			F_1(\frac{u^3}{u^2},\frac{u^4}{u^2})&F_2(\frac{u^3}{u^2},\frac{u^4}{u^2})&F_3(\frac{u^3}{u^2},\frac{u^4}{u^2})&F_4(\frac{u^3}{u^2},\frac{u^4}{u^2})\\F_2(\frac{u^3}{u^2},\frac{u^4}{u^2})&F_3(\frac{u^3}{u^2},\frac{u^4}{u^2})&F_4(\frac{u^3}{u^2},\frac{u^4}{u^2})&0\\F_3(\frac{u^3}{u^2},\frac{u^4}{u^2})&F_4(\frac{u^3}{u^2},\frac{u^4}{u^2})&0&0\\F_4(\frac{u^3}{u^2},\frac{u^4}{u^2})&0&0&0
		\end{bmatrix}
	\end{align}
	for some functions $F_1$, $F_2$, $F_3$, $F_4$ of the variables $z=\frac{u^3}{u^2}$, $w=\frac{u^4}{u^2}$. In particular, $F_1$ is equal to a constant $C_1$ which vanishes whenever $d\neq0$ and from \eqref{F2} and \eqref{Fj} we know that $F_2$, $F_3$, $F_4$ can be expressed as
	\begin{align}
		&F_2(z,w)=-z\,\partial_zf(z,w)-w\,\partial_wf(z,w)-(d-1)\,f(z,w)+C_2\label{4.1_F2}\\
		&F_3(z,w)=\partial_zf(z,w)\label{4.1_F3}\\
		&F_4(z,w)=\partial_wf(z,w)\label{4.1_F4}
	\end{align}
	for some function $f(z,w)$ and some constant $C_2$.	By the flatness conditions, two expressions for $f$ are possible, as shown below. This fully classifies regular four-dimensional Dubrovin-Frobenius manifolds whose operator $L=E\,\circ$ has a single Jordan block.
	\begin{thm}
		The function $f$ realizing \eqref{F2}, \eqref{Fj} is either provided by
		\begin{equation}
			\label{4.1fcaso1}
			f(z,w)=C_3\,w\,e^{C_4\,z}+h(z)
		\end{equation}
		for some constants $C_3$, $C_4$ and some functiton $h(z)$ which is solution to
		\begin{equation}
			\label{eqdiffh}
			h'''(z)-2\,C_4\,h''(z)+C_4^2\,h'(z)+2\,C_3\,C_4\,e^{C_4\,z}=0
		\end{equation}
		or by
		\begin{equation}
			\label{4.1fcaso2}
			f(z,w)=C_3-\frac{A(z)}{2\,B(z)+w}
		\end{equation}
		for some constant $C_3$ and solutions $A(z)$, $B(z)$ to the following system of ODEs:
		\begin{align}
			&A''\,A-(A')^2+2\,\big(C_2+(1-d)\,C_3\big)\,A=0\label{4.1eqA}\\
			&A\,B'''-A'\,\big(B''+1\big)+2\,\big(C_2+(1-d)\,C_3\big)\big(B'+z\big)+C_1=0.\label{4.1eqB}
		\end{align}
	\end{thm}
	\begin{proof}
		By requiring that $R^1_{243}=0$ we get
		\begin{equation}
			\label{eqdipartenza4}
			2\,\partial_wf\,\partial^3_wf-3(\partial^2_wf)^2=0.
		\end{equation}
		Let us distinguish two cases: $\partial^2_wf\ne 0$ and $\partial^2_wf=0$. In the first case we obtain
		\begin{equation}
			\label{4fgen}
			f(z,w)=C(z)-\frac{A(z)}{2\,B(z)+w}
		\end{equation}
		for some functions $A(z)$, $B(z)$, $C(z)$ while in the second one we obtain
		\begin{equation}
			\label{4fsing}
			f(z,w)=w\,h_1(z)+h_2(z)
		\end{equation}
		for some functions $h_1(z)$, $h_2(z)$.
		
		If $f$ is as in \eqref{4fgen} then condition $R^3_{343}=0$ implies that the function $C(z)$ must be equal to a constant $C_3$. Conditions $R^3_{234}=0$ and $R^2_{322}=0$ yields respectively
		\begin{equation}
			\notag
			A''\,A-(A')^2+2\,\big(C_2+(1-d)\,C_3\big)\,A=0
		\end{equation}
		and
		\begin{equation}
			\notag
			A\,B'''-A'\,\big(B''+1\big)+2\,\big(C_2+(1-d)\,C_3\big)\big(B'+z\big)+C_1=0.
		\end{equation}
		All the other conditions in \eqref{EqI}, \eqref{EqII}, \eqref{EqIII}, \eqref{homog}, \eqref{EqVIII}, \eqref{EqXI} hold without imposing more.
		
		If, on the other hand, $f$ is as in \eqref{4fsing}, condition $R^3_{234}=0$ implies that
		\begin{equation}
			\label{h1sing}
			h_1(z)\,h_1''(z)-(h_1'(z))^2=0.
		\end{equation}
		Solutions to \eqref{h1sing} are given by $h_1(z)=C_3\,e^{C_4\,z}$ for some constants $C_3$ and $C_4$, so that 
		\begin{equation}
			\notag
			f(z,w)=C_3\,w\,e^{C_4\,z}+h_2(z).
		\end{equation}
		By imposing condition $R^2_{322}=0$ we get
		\begin{equation}
			\notag
			h_2'''(z)-2\,C_4\,h_2''(z)+C_4^2\,h_2'(z)+2\,C_3\,C_4\,e^{C_4\,z}=0
		\end{equation}
		that yields
		\begin{equation}
			\notag
			h_2(z)=C_7-\frac{e^{C_4\,z}}{C_4^2}\bigg[C_3\,C_4^2\,z^2-C_4\,(2\,C_3+C_5)\,z-C_4\,C_6+2\,C_3+C_5\bigg]
		\end{equation}
		when $C_4\neq0$ and
		\begin{equation}
			\notag
			h_2(z)=C_5\,z^2+C_6\,z+C_7
		\end{equation}
		when $C_4=0$ for some constants $C_5$, $C_6$, $C_7$, so that $f$ becomes respectively
		\begin{equation}
			\label{fsing2_4}
			f(z,w)=C_3\,w\,e^{C_4\,z}+C_7-\frac{e^{C_4\,z}}{C_4^2}\bigg[C_3\,C_4^2\,z^2-C_4\,(2\,C_3+C_5)\,z-C_4\,C_6+2\,C_3+C_5\bigg]
		\end{equation}
		and
		\begin{equation}
			\label{fsing2_4bis}
			f(z,w)=C_3\,w\,e^{C_4\,z}+C_5\,z^2+C_6\,z+C_7.
		\end{equation}
		In both cases it turns out that all the other conditions in \eqref{EqI}, \eqref{EqII}, \eqref{EqIII}, \eqref{homog}, \eqref{EqVIII}, \eqref{EqXI} hold without imposing more. 
	\end{proof}
	\begin{prop}\label{AB}
		The functions $A(z)$ and $B(z)$ appearing in \eqref{4.1eqA} and \eqref{4.1eqB} are expressed via hyperbolic functions and second-order polynomials:
		\begin{eqnarray*}
			A(z)&=&\frac{C_2+(1-d)\,C_3}{C_4^2}\,\sinh^2\big(C_4(z+C_5)\big)\\
			B(z)&=&C_6\cosh(2C_4(z+C_5))+C_7\sinh(2C_4(z+C_5))\\
			&&-\frac{z}{2}\Big(\frac{C_1}{C_2+(1-d)\,C_3}+4C_4C_7\Big)-\frac{z^2}{2}+C_8.
		\end{eqnarray*}
		for some constants $C_4$, $C_5$, $C_6$, $C_7$, $C_8$.
	\end{prop}
	Below flat coordinates are computed for selected other cases, together with some Dubrovin-Frobenius prepotentials.
	\begin{es}
		Let us consider the case \eqref{4.1fcaso1} with $C_3=1$, $C_4=0$ and $d\neq0$. Equation \eqref{eqdiffh} becomes $h'''(z)=0$ yielding $h(z)=a\,z^2+b\,z+c$ for some constants $a,b,c$. In particular we choose $a=c=0$ and $b=1$, so that $h(z)=z$ and $f(z,w)=z+w$. When $d\neq1$, in the flat coordinates
		\begin{align}
			&x^1(u^1,u^2,u^3,u^4)=u^1\notag\\
			&x^2(u^1,u^2,u^3,u^4)=(u^2)^{-d}\big(u^3+u^4\big)\notag\\
			&x^3(u^1,u^2,u^3,u^4)=\frac{1}{2}u^2+u^3\notag\\
			&x^4(u^1,u^2,u^3,u^4)=\frac{1}{1-d}(u^2)^{1-d}\notag
		\end{align}
		we have
		\[\tilde{\eta}=\begin{bmatrix}
			0&1&0&C_2\\1&0&0&0\\0&0&0&1\\C_2&0&1&0
		\end{bmatrix},\,\,
	        \tilde{e}=\tilde{\partial}_1,\,\,
		\tilde{E}=x^1\,\tilde{\partial}_1+(1-d)\,x^2\,\tilde{\partial}_2+x^3\,\tilde{\partial}_3+(1-d)\,x^4\,\tilde{\partial}_4.\]
               Up to second-order polynomial terms, the prepotential is given by
                \[F(x^1,x^2,x^3,x^4)=\frac{C_2}{2}(x^1)^2x^4+x^1x^3x^4-\frac{3}{10}\sqrt[3]{9}(x^4)^\frac{5}{3}+\frac{1}{2}(x^1)^2x^2\]
		when $d=-2$, by
		\[F(x^1,x^2,x^3,x^4)=\frac{1}{4}(x^4)^2\ln{x^4}+\frac{C_2}{2}(x^1)^2x^4+x^1x^3x^4+\frac{1}{2}(x^1)^2x^2\]
		when $d=-1$, by
		\[F(x^1,x^2,x^3,x^4)=\frac{C_2}{2}(x^1)^2x^4+x^1x^3x^4+\frac{1}{2}(x^1)^2x^2\]
		when $d=2$. The case where $d=1$ must be treated separately. In the flat coordinates 
		\begin{align}
			&x^1(u^1,u^2,u^3,u^4)=u^1\notag\\
			&x^2(u^1,u^2,u^3,u^4)=\frac{u^3+u^4}{u^2}\notag\\
			&x^3(u^1,u^2,u^3,u^4)=\frac{1}{2}u^2+u^3\notag\\
			&x^4(u^1,u^2,u^3,u^4)=\ln{u^2}\notag
		\end{align}
		the unit and the Euler vector fields are given by
		\[\tilde{e}=\tilde{\partial}_1,\quad\tilde{E}=x^1\,\tilde{\partial}_1+x^3\,\tilde{\partial}_3+\tilde{\partial}_4.\]
		The metric is as the one for $d\neq1$ and up to second-order polynomial terms the prepotential is
		\[F(x^1,x^2,x^3,x^4)=\frac{1}{8}e^{2x^4}+\frac{C_2}{2}(x^1)^2x^4+x^1x^3x^4+\frac{1}{2}(x^1)^2x^2.\]
	\end{es}
	\begin{es}
		Let us consider the case \eqref{4.1fcaso1} with $C_3=C_4=1$ and $d\neq0$. Equation \eqref{eqdiffh} becomes $h'''(z)-2h''(z)+h'(z)+2e^z=0$ yielding $h(z)=a-(z^2+b\,z+c)e^z$ for some constants $a,b,c$. In particular we choose $a=b=c=0$, so that $h(z)=-z^2e^z$ and $f(z,w)=(w-z^2)\,e^z$. When $d\neq1,2$ the flat coordinates are
		\begin{align}
			&x^1(u^1,u^2,u^3,u^4)=u^1+\frac{u^2}{2(1-d)}\notag\\
			&x^2(u^1,u^2,u^3,u^4)=C_2\ln{u^2}-\frac{2(1-d)(u^2)^2+(u^3)^2-u^2u^4}{(u^2)^2}\,e^\frac{u^3}{u^2}\notag\\
			&x^3(u^1,u^2,u^3,u^4)=(u^2)^{-d-1}\big(u^2u^4-(u^3)^2\big)\,e^\frac{u^3}{u^2}+\frac{C_2(u^2)^{1-d}}{1-d}\notag\\
			&x^4(u^1,u^2,u^3,u^4)=\frac{(u^2)^{2-d}}{2-d}.\notag
		\end{align}
		For $d=-1$ in such coordinates we have
		\[\tilde{e}=\tilde{\partial}_1,\,\,\tilde{E}=x^1\,\tilde{\partial}_1+C_2\,\tilde{\partial}_2+2x^3\,\tilde{\partial}_3+3x^4\tilde{\partial}_4,\,\,
		\tilde{\eta}=\begin{bmatrix}
			0&0&1&0\\0&0&0&-\frac{1}{4}\\1&0&0&0\\0&-\frac{1}{4}&0&0
		\end{bmatrix}
		\]
		and up to second-order polynomial terms the prepotential is given by
		\begin{align}
			\notag
			F(x^1,x^2,x^3,x^4)&=-\frac{1}{32}\sqrt[3]{3}C_2(x^4)^\frac{4}{3}\ln\big({3\,x^4}\big)+\frac{15}{128}\sqrt[3]{3}C_2(x^4)^\frac{4}{3}\\&+\frac{1}{32}\sqrt[3]{9}(x^4)^\frac{2}{3}x^3+\frac{3}{32}\sqrt[3]{3}(x^4)^\frac{4}{3}x^2+\frac{1}{2}(x^1)^2x^3-\frac{1}{4}x^1x^2x^4.\notag
		\end{align}
		For $d=-2$ in flat coordinates we have
		\[\tilde{e}=\tilde{\partial}_1,\,\,\tilde{E}=x^1\,\tilde{\partial}_1+C_2\,\tilde{\partial}_2+3x^3\,\tilde{\partial}_3+4x^4\tilde{\partial}_4,\,\,
		\tilde{\eta}=\begin{bmatrix}
			0&0&1&0\\0&0&0&-\frac{1}{6}\\1&0&0&0\\0&-\frac{1}{6}&0&0
		\end{bmatrix}
		\]
		and up to second-order polynomial terms the prepotential is given by
		\begin{align}
			\notag
			F(x^1,x^2,x^3,x^4)&=-\frac{1}{45}\sqrt{2}C_2(x^4)^\frac{5}{4}\ln\big(2\sqrt{x^4}\big)+\frac{4}{75}\sqrt{2}C_2(x^4)^\frac{5}{4}\\&+\frac{2}{45}\sqrt{2}(x^4)^\frac{5}{4}x^2+\frac{1}{72}\sqrt{x^4}\,x^3+\frac{1}{2}(x^1)^2x^3-\frac{1}{6}x^1x^2x^4.\notag
		\end{align}
		The case $d=2$ must be treated separately. In the flat coordinates 
		\begin{align}
			&x^1(u^1,u^2,u^3,u^4)=u^1-\frac{u^2}{2}\notag\\
			&x^2(u^1,u^2,u^3,u^4)=\frac{2(u^2)^2+u^2u^4-(u^3)^2}{(u^2)^2}\,e^\frac{u^3}{u^2}\notag\\
			&x^3(u^1,u^2,u^3,u^4)=\frac{u^2u^4-(u^3)^2}{(u^2)^3}\,e^\frac{u^3}{u^2}-\frac{C_2}{u^2}\notag\\
			&x^4(u^1,u^2,u^3,u^4)=\ln{u^2}\notag
		\end{align}
		we have
		\[\tilde{e}=\tilde{\partial}_1,\,\,\tilde{E}=x^1\,\tilde{\partial}_1-x^3\,\tilde{\partial}_3+\tilde{\partial}_4,\,\,
		\tilde{\eta}=\begin{bmatrix}
			0&0&1&0\\0&0&0&\frac{1}{2}\\1&0&0&0\\0&\frac{1}{2}&0&C_2
		\end{bmatrix}.
		\]
		Up to second-order polynomial terms the prepotential is
		\[F(x^1,x^2,x^3,x^4)=-\frac{1}{16}x^3\,e^{2x^4}+\bigg(C_2+\frac{x^2}{2}\bigg)\,e^{x^4}+\frac{C_2}{2}x^1(x^4)^2+\frac{1}{2}x^1x^2x^4+\frac{1}{2}(x^1)^2x^3.\]
		In the case where $d=1$, which must be handled separately as well, flat coordinates are given by
		\begin{align}
			&x^1(u^1,u^2,u^3,u^4)=u^1+\frac{u^2}{2}-\frac{u^2}{2}\ln{u^2}\notag\\
			&x^2(u^1,u^2,u^3,u^4)=\frac{u^2u^4-(u^3)^2}{(u^2)^2}\,e^\frac{u^3}{u^2}+C_2\,\ln{u^2}\notag\\
			&x^3(u^1,u^2,u^3,u^4)=\bigg(\frac{u^2u^4-(u^3)^2}{(u^2)^2}\ln{u^2}+2\bigg)\,e^\frac{u^3}{u^2}+\frac{C_2}{2}\big(\ln{u^2}\big)^2\notag\\
			&x^4(u^1,u^2,u^3,u^4)=u^2\notag
		\end{align}
		and 
		\[\tilde{e}=\tilde{\partial}_1,\,\,\tilde{E}=\bigg(x^1-\frac{x^4}{2}\bigg)\,\tilde{\partial}_1+C_2\,\tilde{\partial}_2+x^2\,\tilde{\partial}_3+x^4\,\tilde{\partial}_4,\,\,
		\tilde{\eta}=\begin{bmatrix}
			0&1&0&0\\1&0&0&0\\0&0&0&\frac{1}{2}\\0&0&\frac{1}{2}&0
		\end{bmatrix}.
		\]
		Up to second-order polynomial terms the prepotential is
		\begin{align}
			\notag
			F(x^1,x^2,x^3,x^4)&=\frac{C_2}{24}(x^4)^2\big(\ln{x^4}\big)^3-\frac{3}{16}\bigg(C_2+\frac{2}{3}x^2\bigg)(x^4)^2\big(\ln{x^4}\big)^2\\&+\frac{7}{16}\bigg(C_2+\frac{6}{7}x^2+\frac{4}{7}x^3\bigg)(x^4)^2\ln{x^4}+\frac{1}{2}x^1x^3x^4\notag\\&+\frac{1}{2}(x^1)^2x^2-\frac{7\,x^2+6\,x^3}{16}(x^4)^2.\notag
		\end{align}
	\end{es}

\section{Conclusions}
In this paper we have studied regular Dubrovin-Frobenius manifold structures $(\eta,\circ,e,E)$ assuming that the Jordan form of 
 the operator of multiplication by the Euler vector field $L$ contains a single Jordan block.
We have used a special set of coordinates introduced by David and Hertling where the operator $L$ and the Dubrovin-Frobenius metric take the form
\[L=\begin{bmatrix}u^1&0&0&\dots&0&0\\u^2&u^1&0&\dots&0&0\\u^3&u^2&u^1&\dots&0&0\\\vdots&\vdots&\vdots&\ddots&\vdots&\vdots\\u^{n-1}&u^{n-2}&u^{n-3}&\dots&u^1&0\\u^n&u^{n-1}&u^{n-2}&\dots&u^2&u^1
		\end{bmatrix},\eta=\begin{bmatrix}
			\d_1 H&\d_2 H&\d_3 H&\dots&\d_{n-1}H&\d_nH\\\d_2H&\d_3H&\d_4H&\dots&\d_nH&0\\\d_3H&\d_4H&\d_5H&\dots&0&0\\\vdots&\vdots&\vdots&\ddots&\vdots&\vdots\\\d_{n-1}H&\d_nH&0&\dots&0&0\\\d_nH&0&0&\dots&0&0
		\end{bmatrix}.\]
We have shown that, up to constants, the metric potential $H$ has the form 
\begin{equation}
				H=(u^2)^{1-d}\,f+C_2\,\varphi(u^2)+C_1\,u^1
			\end{equation}
		where
			\begin{equation}
				\varphi(u^2)=\begin{cases}
					\frac{(u^2)^{1-d}}{1-d}\qquad&\text{if }d\neq1\\
					\ln{u^2}\qquad&\text{if }d=1
				\end{cases}
		\end{equation}
and $f=f(z^1,...,z^{n-2})$ with $z^i=\frac{u^{i+2}}{u^2}$ for each $i=1,\dots,n-2$.
In dimension $n=2,3,4$ we have obtained explicit formulas for the metric potential $H$ (see Table 1 below) and, in some special cases, we have computed the flat coordinates and the corresponding prepotential. 
\begin{table}[h]
\begin{center}
\begin{tabular}{|p{15mm}|p{125mm}|}
\hline
$n=2$ & $f=C_3$   \\ 
\hline
$n=3$ & $f=C_3\,z+C_4$  or $f=\frac{C_4}{z+C_3}+C_5$. \\
\hline
$n=4$ & $f=C_3we^{C_4\,z}+C_7-\frac{e^{C_4\,z}}{C_4^2}\bigg[C_3C_4^2z^2-C_4(2C_3+C_5)z-C_4C_6+2C_3+C_5\bigg]$,
\newline
or
\newline
$f=C_3we^{C_4\,z}+C_5\,z^2+C_6\,z+C_7$,
\newline
or
\newline
$f=C_3-\frac{A(z)}{2\,B(z)+w}$
\newline
with $z=\frac{u^{3}}{u^2}$, $w=\frac{u^{4}}{u^2}$, $A(z)$ and $B(z)$ defined as in Proposition \ref{AB}.
\\
\hline
\end{tabular}
\end{center}
\caption{\footnotesize Metric potential}
\label{tab_results}
\end{table} 
\newline
According to the general theory the metrics $\eta^{-1}$ and $L\eta^{-1}$ define a flat pencil of metrics.
 Therefore a byproduct of our results is a list of non semisisimple flat
 pencils of metrics that define the bi-Hamiltonian structures of the principal hierarchies of the associated Dubrovin-Frobenius manifolds. The study of this class of bi-Hamiltonian structures and of their bi-Hamiltonian deformations in the non semisimple case is at a preliminary stage and only a few results are available so far (see for instance \cite{DLS}).
 The explicit examples obtained in this work might be the starting point of future investigations.
\newline
\newline
\noindent{\bf Acknowledgements}.  P.~L. is supported by funds of H2020-MSCA-RISE-2017 Project No. 778010 IPaDEGAN.
 Authors are also thankful to GNFM - INdAM for supporting activities that contributed to the research reported in this paper.

\end{document}